\documentclass[12pt,reqno]{article}

\usepackage[usenames]{color}
\usepackage{amssymb}
\usepackage{amsmath}
\usepackage{amsthm}
\usepackage{amsfonts}
\usepackage{amscd}
\usepackage{graphicx}

\usepackage[colorlinks=true,
linkcolor=webgreen,
filecolor=webbrown,
citecolor=webgreen]{hyperref}

\definecolor{webgreen}{rgb}{0,.5,0}
\definecolor{webbrown}{rgb}{.6,0,0}

\usepackage{color}
\usepackage{fullpage}
\usepackage{float}

\usepackage{graphics}
\usepackage{latexsym}

\DeclareMathOperator{\arctanh}{arctanh}

\begin{document}
	
	\theoremstyle{plain}
	\newtheorem{theorem}{Theorem}
	\newtheorem{corollary}[theorem]{Corollary}
	\newtheorem{lemma}{Lemma}
	\newtheorem{example}{Examples}
	\newtheorem*{remark}{Remark}
	
	\begin{center}
		\vskip 1cm
		{\LARGE\bf
			On some Series involving Reciprocals of $\binom{2n}{n}$ and the Catalan's Constant $G$}
		
		\vskip 1cm
		
		{\large
			Olofin Akerele$^1$ \\
		Department of Mathematics, \\University of Ibadan, Ibadan, Nigeria \\
		\href{mailto:akereleolofin@gmail.com}{\tt akereleolofin@gmail.com}

			\vskip 0.2 in
			
			Quadri Adeshina$^2$ \\
			Department of Mathematics, \\University of Ibadan, Ibadan, Nigeria \\
			\href{mailto:adeshinaquadri69@gmail.com}{\tt adeshinaquadri69@gmail.com}

		}
		
	\end{center}
	
	\vskip .2 in
	
	\begin{abstract}
	 We investigate a class of combinatorial sums involving reciprocals of central binomial coefficients , employing generating functions as the primary solution technique to formulate and analyze series involving the Catalan's constant. Using a direct approach, we derive new identities through integral techniques.
	\end{abstract}
	
	
	
	\bigskip

	\section{Introduction}
\par In the field of mathematics, Catalan's constant, denoted as 
$G$, represents a fascinating quantity characterized by its unique properties. Specifically, it is defined as the alternating sum of the reciprocals of the odd square numbers. This can be mathematically expressed as follows:
$$G:=\sum_{n=0}^{\infty}\frac{(-1)^n}{(2n+1)^2}\approx 0.915965..$$
This intriguing constant arises in various mathematical contexts, particularly in number theory and combinatorial mathematics, where it serves as a fundamental element in the study of series and special functions. The behavior of Catalan's constant reveals deep connections to other mathematical constants and functions, showcasing the rich interplay within the realm of mathematical analysis.
\newpage
In (\cite{David},Pg 2), we have that,
$$G=\int_0^1\frac{\arctan x}{x}\,\mathrm{d}x=\frac{1}{2}\int_0^{\infty}\frac{x}{\cosh x}\,\mathrm{d}x=2\int_0^{\frac{\pi}{4}}\log(2\cos x)\,\mathrm{d}x$$
which serves as some basic integral representation of $G$.\par Now, the binomial coefficient $\binom{n}{m}$ is defined by $$\binom{n}{m}=\begin{cases}
\frac{n!}{m!(n-m)!}, & \text{if} \ n\geq m;\\
0,                   & \text{if} \ n<m.
\end{cases}$$
where $n$ and $m$ are non-negative integers. For some results involving the inverse of binomial coefficients, see \cite{Batir},\cite{Batirr},\cite{Lehmer}.\par Among the conclusions drawn in this paper, we will ascertain that if $n$ is a non-negative integer then 	$$\sum_{n=1}^{\infty}\binom{2n}{n}\binom{2n+2}{n+1}\frac{1}{16^n(n+1)(2n+3)}=\frac{22}{3}-\frac{8}{\pi}-\frac{16G}{\pi}$$
Throughout this paper, we verify our results using Computer Algebra System (CAS) software Mathematica 13.3.
\section{Generating Functions}
\par This section provides essential preliminary concepts that will serve as foundational building blocks for the analyses and results presented in the subsequent section.\par By establishing these key Lemmas, we aim to create a structured basis that will facilitate a clearer understanding of the subsequent results to follow. We proceed as follows:
\begin{lemma}
	\label{lemma 1}
For $x\in(0,4)$, then
\begin{equation}
	\label{1}
	\sum_{n=1}^{\infty}\frac{x^n}{n\binom{2n}{n}}=2\sqrt{\frac{x}{4-x}}\arctan\sqrt{\frac{x}{4-x}}
\end{equation}
\end{lemma}
\begin{proof}
	Observe that,
\begin{align*}
	\sum_{n=1}^{\infty}\frac{x^n}{n\binom{2n}{n}}&=\frac{1}{2}\sum_{n=1}^{\infty}x^n\int_0^{\frac{\pi}{2}}2\sin^{2n-1}\theta\cos^{2n-1}\theta\,\mathrm{d}\theta=\int_0^{\frac{\pi}{2}}\frac{1}{\sin\theta\cos\theta}\sum_{n=1}^{\infty}(x\sin^2\theta\cos^2\theta)^n\,\mathrm{d}\theta\\
	&=\int_0^{\frac{\pi}{2}}\frac{x\sin\theta\cos\theta}{1-x\sin^2\theta\cos^2\theta}\,\mathrm{d}\theta = \frac{x}{2}\int_0^{\frac{\pi}{2}}\frac{\sin 2\theta}{1-\frac{x}{4}\sin 2\theta}\,\mathrm{d}\theta.
\end{align*}
Set $p=\cos 2\theta$ and the result follows.
\end{proof}
\begin{lemma}
	For all $x\in(0,4]$, then
	\begin{equation}
		\label{lemma 2}
	\sum_{n=1}^{\infty}\frac{x^n}{n^2\binom{2n}{n}}	=2\arctan^2\left(\sqrt{\frac{x}{4-x}}\right)
	\end{equation}
\end{lemma}
\begin{proof}
	Notice that we can write as follows;
\begin{align*}
\sum_{n=1}^{\infty}\frac{x^n}{n^2\binom{2n}{n}}	=\frac{1}{2}\sum_{n=1}^{\infty}\int_0^1\frac{x^n(w-w^2)^{n-1}}{n}\,\mathrm{d}w=\frac{1}{2}\int_0^1\frac{1}{w(1-w)}\sum_{n=1}^{\infty}\frac{(wx(1-w))^2}{n}\,\mathrm{d}w
\end{align*}
Note that $-\log(1-x)=\sum_{n=1}^{\infty}\frac{x^n}{n}$, thus we have	
\begin{align*}
\frac{1}{2}\int_0^1\frac{1}{w(1-w)}\sum_{n=1}^{\infty}\frac{(wx(1-w))^2}{n}\,\mathrm{d}w=-\frac{1}{2}\underbrace{\int_0^1\frac{\log(xw^2-xw+1)}{w(1-w)}\,\mathrm{d}w}_{J(x)}
\end{align*}
Observe,
\begin{align*}
	J'(x)=2\int_0^1\frac{\partial}{\partial x}\frac{\log(xw^2-xw+1)}{w}\,\mathrm{d}w=-\frac{4}{x}\sqrt{\frac{x}{4-x}}\arctan\sqrt{\frac{x}{4-x}}
\end{align*}
Integrating both sides, we obtain the desired result.
\end{proof}
\begin{lemma}
	\label{lemma 3}
	For all $x\in(0,4)$, then
\begin{equation}
	\sum_{n=1}^{\infty}\frac{x^{n+3/2}}{n(n+3/2)\binom{2n}{n}}=	\frac{4}{9}\sqrt{4-x}\left(\sqrt{\frac{x}{4-x}}(x+24)-3(x+8)\arctan\sqrt{\frac{x}{4-x}}\right)	
\end{equation}
\end{lemma}
\begin{proof}
From Lemma \ref{lemma 1}, multiply both sides by $\sqrt{x}$ and then integrate with respect to $x$. Thus, the result follows.
\end{proof}

\par Numerous authors have proposed similar generating functions expressed in terms of the arcsine function, but the difference is very minimal, (See, \cite{Lehmer}).
\begin{lemma}
	\label{lemma 4}
For all $x\in[-1,1]$
\begin{equation}
	\sum_{n=1}^{\infty}\binom{2n}{n}\frac{x^{2n+2}}{2^{2n+1}(n+1)(2n+3)}=1-\frac{x^2}{6}-\frac{\sqrt{1-x^2}}{2}-\frac{\arcsin x}{2x}
\end{equation}
\end{lemma}
\begin{proof}
	It can be shown that $$\sum_{n=1}^{\infty}\binom{2n}{n}\frac{x^{2n}}{4^n}=\frac{1}{\sqrt{1-x^2}}-1$$
Thus, multiply both sides of the above identity by $x^2$ and then integrate both sides to get,
$$\sum_{n=1}^{\infty}\binom{2n}{n}\frac{x^{2n+3}}{(2n+3)4^n}=\frac{1}{6}(3\arcsin x-3x\sqrt{1-x^2}-2x^3)+C$$
Observe, as $x\to 0$ we have that $C\to 0$. Hence $C=0$. By dividing through by $x^2$ and integrating both sides, we obtain;
$$	\sum_{n=1}^{\infty}\binom{2n}{n}\frac{x^{2n+2}}{2^{2n+1}(n+1)(2n+3)}=-\frac{x^2}{6}-\frac{\sqrt{1-x^2}}{2}-\frac{\arcsin x}{2x}+C_1$$
Similarly as $x\to 0$, we have that $C_1\to 1$. Hence, the result follows directly.
\end{proof}
\begin{lemma}
	\label{lemma 5}
For all $x\in [-1,1]$
$$\sum_{n=1}^{\infty}\binom{2n}{n}\frac{x^{2n+2}}{16^n(n+1)^2(2n+1)}=\log 2-2+x\arcsin x+2\sqrt{1-x^2}-\log x-\arctanh(\sqrt{1-x^2})-\frac{x^2}{4}$$
\end{lemma}
\begin{proof}
	Since, we can show that 
	$$\sum_{n=1}^{\infty}\binom{2n}{n}\frac{x^{2n+1}}{4^n(2n+1)}=\arcsin x-x$$
	Integrating both sides of the above identity we get a new identity 
	$$\sum_{n=1}^{\infty}\binom{2n}{n}\frac{x^{2n+2}}{4^n(2n+1)(2n+2)}=x\arcsin x+\sqrt{1-x^2}-\frac{x^2}{2}-1$$
	Dividing both sides by $x$ and integrating both sides. The desired result follows immediately.
\end{proof}
\begin{lemma}
	\label{lemma 6}
	For all $x\in[-1,1]$
$$\sum_{n=1}^{\infty}\binom{2n}{n}\frac{x^{2n+2}}{2^{2n+1}(n+1)(2n+1)(2n+3)}=\frac{1}{12}\left(9\sqrt{1-x^2}+\frac{(6x^2+3)}{x}\arcsin x-2x^2-12\right)$$
\end{lemma}

\begin{proof}
	Notice, 
	$$\sum_{n=1}^{\infty}\binom{2n}{n}\frac{x^{2n+2}}{2^{2n+1}(n+1)(2n+1)}=x\arcsin x+\sqrt{1-x^2}-\frac{x^2}{2}-1$$
The result follows from the above identity.
\end{proof}
\section{Main Results}
\begin{theorem}
	If $n$ is a non-negative integer, then we have
	\begin{equation}
		\sum_{n=1}^{\infty}\binom{2n}{n}\binom{2n+2}{n+1}\frac{1}{16^n(n+1)(2n+3)}=\frac{22}{3}-\frac{8}{\pi}-\frac{16G}{\pi}
	\end{equation}
\end{theorem}
\begin{proof}
	From Lemma \ref{lemma 4}, set $x=\sin y$ and using the walli's integral formula;
	$$\binom{2n}{n}=\frac{2}{\pi}\int_0^{\frac{\pi}{2}}2^{2n}\sin^{2n}t\,\mathrm{d}t$$
	Thus, 
	\begin{align*}
	\sum_{n=1}^{\infty}\binom{2n}{n}\binom{2n+2}{n+1}\frac{\pi}{2^{4n+4}(n+1)(2n+3)}&=\int_0^{\frac{\pi}{2}}\left(1-\frac{\sin^2x}{6}-\frac{\cos x}{2}-\frac{x}{2\sin x}\right)\,\mathrm{d}x\\&=\frac{11\pi}{24}-\frac{1}{2}-G
\end{align*}
Since $\int_0^{\pi/2}\frac{x}{2\sin x}\,\mathrm{d}x=G$, check [\cite{David},Pg 2].
The result follows immediately.
\end{proof}
\begin{theorem}
	If n is a non-negative integer, then 
\begin{equation}
	\sum_{n=1}^{\infty}\binom{2n}{n}\binom{2n+2}{n+1}\frac{1}{64^n(n+1)^2(2n+1)}=8\log 2-\frac{17}{2}+\frac{24}{\pi}-\frac{16 G}{\pi}
\end{equation}
\end{theorem}
\begin{proof}
From Lemma \ref{lemma 5}, Set $x=\sin y$, while integrating from $0$ to $\pi/2$ and using the identity; 
\begin{equation}
	\label{7}
\frac{\pi}{2^{2n+3}}\binom{2n+2}{n+1}=\int_0^{\pi/2}\sin^{2n+2}t\,\mathrm{d}t
\end{equation}
Note that the above equality follows directly from the walli's integral formula. The proof is straightforward from this end.
\end{proof}
\begin{theorem}
	If $n$ is a non-negative integer, then
	\begin{equation}
\sum_{n=1}^{\infty}\binom{2n}{n}\binom{2n+2}{n+1}\frac{1}{16^n(n+1)(2n+1)(2n+3)}=\frac{20}{\pi}+\frac{8G}{\pi}-\frac{26}{3}	
	\end{equation}
\end{theorem}
\begin{proof}
From Lemma \ref{lemma 6}, following same pattern in Theorem 2 and using \eqref{7} with the famous identity from [\cite{David},Pg 2]; 
$$\int_0^{\pi/2}\frac{x}{2\sin x}\,\mathrm{d}x=G$$
The desired result follows immediately.
\end{proof}
\section{Some Interesting Series}
From Lemma \ref{lemma 1} to Lemma \ref{lemma 3}, we can generate some Lehmer Series [check, \cite{Lehmer}]
\begin{align}
	&\sum_{n=1}^{\infty}\frac{2^n}{n}\binom{2n}{n}^{-1}=\frac{\pi}{2}\\
	&\sum_{n=1}^{\infty}\frac{4^n}{n^2}\binom{2n}{n}^{-1}=\frac{\pi^2}{2}\\
	&\sum_{n=1}^{\infty}\frac{1}{n^2(n+2)}\binom{2n}{n}^{-1}=\frac{1}{72}(-117+42\pi\sqrt{3}-10\pi^2)\\
	&\sum_{n=1}^{\infty}\frac{1}{n(n+3/2)}\binom{2n}{n}^{-1}=\frac{100}{9}-2\sqrt{3}\pi\\
	&\sum_{n=1}^{\infty}\frac{1}{n(n+3/2)(n+5/2)}\binom{2n}{n}^{-1}=\frac{4}{225}(-1384+255\pi\sqrt{3})
\end{align}
We also have, 
\begin{align}
	\label{12}
   &\sum_{n=1}^{\infty}\frac{n^2}{16^n(2n-1)^2(2n+1)}\binom{2n}{n}^2 = \frac{G}{4\pi}+\frac{1}{8\pi}
\end{align}
It's easy to derive \eqref{12} from 
$$\sum_{n=1}^{\infty}\frac{2n^2x^{2n-1}}{4^n(2n-1)^2(2n+1)}\binom{2n}{n}=\frac{1}{8x^2}\left((2x^2+1)\arcsin x-x\sqrt{1-x^2}\right)$$
Now multiply by $x$ and set $\sin t=x$, then we integrate both sides from 0 to $\pi/2$. Then, \eqref{12} follows directly.
\par From Lemma \ref{lemma 3}, we noticed that the Guass Hyper-geometric function of the form \\ $_2F_1[2,2;9/2;1/4k]$ can be obtained. Observe, 
\begin{align*}
\sum_{n=1}^{\infty}\frac{1}{n(n+3/2)}\binom{2n}{n}^{-1}&=\frac{1}{105} \ _2F_1\left[2,2;\frac{9}{2};\frac{1}{4}\right]+\frac{210\sqrt{3}}{105}\pi-\frac{1120}{105}=\frac{100}{9}-2\sqrt{3}\pi
\end{align*}
\begin{align*}
\sum_{n=1}^{\infty}\frac{1}{2^{n+3/2}n(n+3/2)}\binom{2n}{n}^{-1}&=\frac{1}{840\sqrt{2}} \ _2F_1\left[2,2;\frac{9}{2};\frac{1}{8}\right]+\frac{11760\sqrt{7}}{840\sqrt{2}}\arcsin\frac{1}{2\sqrt{2}}-\frac{11200}{840\sqrt{2}}\\&=\frac{\sqrt{14}}{9}\left(7\sqrt{7}-51\arctan\frac{1}{\sqrt{7}}\right)\\
\sum_{n=1}^{\infty}\frac{1}{3^{n+3/2}n(n+3/2)}\binom{2n}{n}^{-1}&=\frac{1}{2835\sqrt{3}} \ _2F_1\left[2,2;\frac{9}{2};\frac{1}{12}\right]+\frac{41580\sqrt{11}}{2835\sqrt{3}}\arcsin\frac{1}{2\sqrt{3}}-\frac{40320}{2835\sqrt{3}}\\&=\frac{4}{27\sqrt{3}}\left(73-75\sqrt{11}\arctan\frac{1}{\sqrt{11}}\right)
\end{align*}
From the above, we can see that;
\begin{align}
	_2F_1\left[2,2;\frac{9}{2};\frac{1}{4}\right]&=1120-210\sqrt{3}\pi+105\left(\frac{100}{9}-2\pi\sqrt{3}\right)\approx 1.2947\\
	_2F_1\left[2,2;\frac{9}{2};\frac{1}{8}\right]&=11200-11760\sqrt{7}\arcsin\frac{1}{2\sqrt{2}}+\frac{560\sqrt{7}}{3}\left(7\sqrt{7}-51\arctan\frac{1}{\sqrt{7}}\right)\approx 1.12378\\
 	_2F_1\left[2,2;\frac{9}{2};\frac{1}{12}\right]&=40320-41580\sqrt{11}\arcsin\frac{1}{2\sqrt{3}}+420\left(73-75\sqrt{11}\arctan\frac{1}{\sqrt{11}}\right)\approx 1.0795	
\end{align}
\par In light of the aforementioned conclusions, a distinct pattern is discernible. Consequently, if $k$ is a natural number, the following conjecture is proposed to hold true.
$$_2F_1\left[2,2;\frac{9}{2};\frac{1}{4k}\right]=a_k+\sqrt{4k-1}\left(b_k\arcsin\left(\frac{1}{2\sqrt{k}}\right)+c_k\arctan\left(\frac{1}{\sqrt{4k-1}}\right)\right)$$
Where the general expression of $a_k,b_k$, and $ c_k$ remains open.

\end{document}